\documentclass[a4paper,oneside,10pt,final,reqno]{amsart}

\usepackage{amssymb}
\usepackage[mathcal]{euscript}
\usepackage{mathrsfs}
\usepackage[frame,cmtip,curve,arrow,matrix,line,graph]{xy}
\usepackage{enumerate}

\usepackage[dvipdfmx]{hyperref}

\usepackage{version,color}
\usepackage[l2tabu,orthodox]{nag}
\usepackage{showkeys}
\usepackage[all, warning]{onlyamsmath}

\numberwithin{equation}{section}
\allowdisplaybreaks

\theoremstyle{definition}
\newtheorem{thm}{Theorem}[section]
\newtheorem{prop}[thm]{Proposition}
\newtheorem{dfn}[thm]{Definition}
\newtheorem{lem}[thm]{Lemma}
\newtheorem{fct}[thm]{Fact}

\newtheorem*{dfn*}{Definition}
\newtheorem*{rmk*}{Remark}
\newtheorem*{fct*}{Fact}

\newcommand{\wh}{\widehat}

\newcommand{\inj}{\hookrightarrow}
\newcommand{\surj}{\twoheadrightarrow}
\newcommand{\simto}{\xrightarrow{\sim}}

\newcommand{\longsimto}{\xrightarrow{\ \sim \ }}

\newcommand{\tsum}{{\textstyle{\sum}}}

\newcommand{\bbC}{\mathord{\mathbb{C}}}
\newcommand{\bbF}{\mathord{\mathbb{F}}}
\newcommand{\bbN}{\mathord{\mathbb{N}}}
\newcommand{\bbQ}{\mathord{\mathbb{Q}}}

\newcommand{\bbZ}{\mathord{\mathbb{Z}}}
\newcommand{\calE}{\mathord{\mathcal{E}}}
\newcommand{\calL}{\mathord{\mathcal{L}}}
\newcommand{\calP}{\mathord{\mathcal{P}}}
\newcommand{\frka}{\mathord{\mathfrak{a}}}
\newcommand{\bfR}{\mathord{\mathbf{R}}}

\newcommand{\shO}{\mathord{\mathcal{O}}}
\newcommand{\catA}{\mathord{\mathfrak{A}}}
\newcommand{\catB}{\mathord{\mathfrak{B}}}
\newcommand{\catC}{\mathord{\mathfrak{C}}}
\newcommand{\modc}{\mathord{\mathfrak{mod}}}
\newcommand{\frkS}{\mathop{\mathfrak{S}}\nolimits}
\newcommand{\frkT}{\mathop{\mathfrak{T}}\nolimits}
\newcommand{\Sky}{\mathop{\mathfrak{Sky}}\nolimits}
\newcommand{\Coh}{\mathop{\mathfrak{Coh}}\nolimits}
\newcommand{\DCoh}{\mathop{D^b\mathfrak{Coh}}\nolimits}
\newcommand{\DFuk}{\mathop{D\mathfrak{Fuk}}\nolimits}

\newcommand{\EH}{\mathop{\mathbb{U}}\nolimits}

\newcommand{\im}{\mathop{\mathrm{im}}\nolimits}
\newcommand{\coim}{\mathop{\mathrm{coim}}\nolimits}
\newcommand{\supp}{\mathop{\mathrm{supp}}\nolimits}

\newcommand{\SL}{\mathop{\mathrm{SL}}\nolimits}
\newcommand{\Sk}{\mathop{\mathrm{Sk}}\nolimits}
\newcommand{\Aut}{\mathop{\mathrm{Aut}}\nolimits}
\newcommand{\Ext}{\mathop{\mathrm{Ext}}\nolimits}
\newcommand{\Hom}{\mathop{\mathrm{Hom}}\nolimits}
\newcommand{\Iso}{\mathop{\mathrm{Iso}}\nolimits}
\newcommand{\Hall}{\mathop{\mathrm{Hall}}\nolimits}
\newcommand{\DHall}{\mathop{\mathrm{DHall}}\nolimits}
\newcommand{\Jac}{\mathop{\mathrm{Jac}}\nolimits}
\newcommand{\Spec}{\mathop{\mathrm{Spec}}\nolimits}
\newcommand{\Spf}{\mathop{\mathrm{Spf}}\nolimits}

\newcommand{\mon}{\text{mon}}
\newcommand{\Tate}{\text{Tate}}
\newcommand{\torus}{\text{torus}}

\newcommand{\mnd}[1]{\langle\!\langle #1 \rangle\!\rangle}

\begin{document}


\title{Elliptic Hall algebra on $\mathbb{F}_1$}

\author{Shintarou Yanagida}
\address{Graduate School of Mathematics, Nagoya University 
Furocho, Chikusaku, Nagoya, Japan, 464-8602.}
\email{yanagida@math.nagoya-u.ac.jp}

\date{August 30, 2017}

\begin{abstract}
We construct Hall algebra of elliptic curve over $\mathbb{F}_1$
using the theory of monoidal scheme due to Deitmar
and the theory of Hall algebra 
for monoidal representations due to Szczesny.
The resulting algebra is shown to be a specialization of 
elliptic Hall algebra studied by Burban and Schiffmann.
Thus our algebra is isomorphic to 
the skein algebra for torus
by the recent work of Morton and Samuelson.
\end{abstract}

\maketitle

\section{Introduction}

This note is motivated to understand 
in terms of categorical viewpoint
the recent work of Morton and Samuelson \cite{MS}
establishing the algebra isomorphism between 
the torus skein algebra and the elliptic Hall algebra.

In the late 1980s, Turaev \cite{Turaev:ASMP, Turaev}
introduced the skein algebra $\Sk(\Sigma)$
as a $q$-deformation of 
the Goldman Lie algebra \cite{Goldman}
on an oriented surface $\Sigma$.
The Goldman Lie algebra is the Lie algebra 
encoding the symplectic structures of 
character varieties on $\Sigma$.
Recently, Morton and Samuelson \cite{MS}
discovered remarkable
relationship between skein algebra for a torus $T$ 
and the Ringel-Hall algebra for an elliptic curve $E$.
Let us start with this introduction by explaining 
Turaev's skein algebra.


\subsection{Skein algebra}
\label{subsec:skein}

Let $R:=\bbZ[s^{\pm1},v^{\pm1}]$.
A skein module of an oriented 3-fold $M$ is 
the quotient of the $R$-module spanned by 
the isotopy classes of framed oriented links in $M$
by the skein relation $(\text{sk})$ 
and the framing relation $(\text{fr})$.
\begin{align*}
S(M) := 
 \left<
 \begin{array}{l}
 \text{isotopy classes of} \\
 \text{framed oriented links in $M$} 
 \end{array}
 \right>_{R\text{-lin}}
 \Big/ (\text{sk}),(\text{fr})
\end{align*}
The skein relation $(\text{sk})$ 
is defined 
as the following local diagram.

\hskip6em
{\unitlength 0.1in%
\begin{picture}(28.0,4.0)(-6.0,-6.0)%
\special{pn 13}%
\special{pa 1200 600}%
\special{pa 1350 450}%
\special{fp}%
\special{pn 13}%
\special{pa 200 600}%
\special{pa 600 200}%
\special{fp}%
\special{sh 1}%
\special{pa 600 200}%
\special{pa 539 233}%
\special{pa 562 238}%
\special{pa 567 261}%
\special{pa 600 200}%
\special{fp}%
\special{pn 13}%
\special{pa 600 600}%
\special{pa 450 450}%
\special{fp}%
\special{pn 13}%
\special{pa 350 350}%
\special{pa 200 200}%
\special{fp}%
\special{sh 1}%
\special{pa 200 200}%
\special{pa 233 261}%
\special{pa 238 238}%
\special{pa 261 233}%
\special{pa 200 200}%
\special{fp}%
\special{pn 13}%
\special{pa 1600 600}%
\special{pa 1200 200}%
\special{fp}%
\special{sh 1}%
\special{pa 1200 200}%
\special{pa 1233 261}%
\special{pa 1238 238}%
\special{pa 1261 233}%
\special{pa 1200 200}%
\special{fp}%
\special{pn 13}%
\special{pa 1450 350}%
\special{pa 1600 200}%
\special{fp}%
\special{sh 1}%
\special{pa 1600 200}%
\special{pa 1539 233}%
\special{pa 1562 238}%
\special{pa 1567 261}%
\special{pa 1600 200}%
\special{fp}%
\special{pn 13}%
\special{ar 2800 400 100 200 4.7123890 1.5707963}%
\special{fp}%
\special{sh 1}%
\special{pa 2875 410}%
\special{pa 2925 410}%
\special{pa 2900 340}%
\special{pa 2875 410}%
\special{fp}%
\special{pn 13}%
\special{ar 3200 400 100 200 1.5707963 4.7123890}%
\special{fp}%
\special{sh 1}%
\special{pa 3075 410}%
\special{pa 3125 410}%
\special{pa 3100 340}%
\special{pa 3075 410}%
\special{fp}%
\put(-5.00,-5.00){\makebox(0,0)[lb]{{\large(sk)}}}%
\put( 8.00,-4.80){\makebox(0,0)[lb]{{\large $-$}}}%
\put(18.00,-4.70){\makebox(0,0)[lb]{{\large $=$}}}%
\put(21.00,-5.00){\makebox(0,0)[lb]{{\large $(s-s^{-1})$}}}%
\end{picture}}%

\noindent
We omit the definition of the framing relation.
It concerns with twists of framed links, 
and the variable $v$ counts the number of twists.

Let $I$ denote the unit interval.
The skein algebra of oriented surface $\Sigma$ is 
$\Sk(\Sigma) := S(\Sigma \times I)$ as an $R$-module
with the multiplication given by
placing one copy of $\Sigma \times I$ 
on top of another $\Sigma \times I$.

For example, if $\Sigma$ is a torus,
we set $L_{0,1}$, $L_{1,1}$ as in the diagrams below.
The square shows the 

\hskip 12em
$L_{0,1}$
\hskip 6.5em
$L_{1,1}$
\hskip 5.5em
$L_{1,1} \cdot L_{0,1}$

\hskip10em
{\unitlength 0.1in%
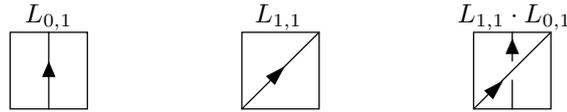
\begin{picture}(10.0,4.00)(2.0,-8.00)%
\special{pn 8}%
\special{pa 400 400}%
\special{pa 400 800}%
\special{pa 800 800}%
\special{pa 800 400}%
\special{pa 400 400}%
\special{fp}%
\special{pn 8}%
\special{pa 600 400}%
\special{pa 600 800}%
\special{fp}%
\special{sh 1}%
\special{pa 570 630}%
\special{pa 630 630}%
\special{pa 600 560}%
\special{pa 570 630}%
\special{fp}%
\special{pn 8}%
\special{pa 1600 400}%
\special{pa 1600 800}%
\special{pa 2000 800}%
\special{pa 2000 400}%
\special{pa 1600 400}%
\special{fp}%
\special{pn 8}%
\special{pa 1600 800}%
\special{pa 2000 400}%
\special{fp}%
\special{sh 1}%
\special{pa 1815 585}%
\special{pa 1770 670}%
\special{pa 1730 630}%
\special{pa 1815 585}%
\special{fp}%
\special{pn 8}%
\special{pa 2800 400}%
\special{pa 2800 800}%
\special{pa 3200 800}%
\special{pa 3200 400}%
\special{pa 2800 400}%
\special{fp}%
\special{pn 8}%
\special{pa 2800 800}%
\special{pa 3200 400}%
\special{fp}%

\special{pn 8}%
\special{pa 3000 400}%
\special{pa 3000 550}%
\special{fp}%

\special{pn 8}%
\special{pa 3000 650}%
\special{pa 3000 800}%
\special{fp}%
\special{pn 8}%
\special{sh 1}%
\special{pa 2945 660}%
\special{pa 2900 745}%
\special{pa 2860 705}%
\special{pa 2945 660}%
\special{fp}%
\special{pn 8}%
\special{sh 1}%
\special{pa 2970 510}%
\special{pa 3030 510}%
\special{pa 3000 440}%
\special{pa 2970 510}%
\special{fp}%
\end{picture}}%

\noindent
Then the skein relation implies 
the following equation.
The corresponding diagrams are denoted below
with orientations on links omitted.

\hskip11em
$L_{1,1} \cdot L_{0,1} \  = \ 
 L_{0,1} \cdot L_{1,1}+(s-s^{-1})L_{1,2}$.

\hskip10em
{\unitlength 0.1in%
\begin{picture}(10.0,4.00)(2.0,-8.00)%
\special{pn 8}%
\special{pa 400 400}%
\special{pa 400 800}%
\special{pa 800 800}%
\special{pa 800 400}%
\special{pa 400 400}%
\special{fp}%
\special{pn 8}%
\special{pa 400 800}%
\special{pa 800 400}%
\special{fp}%
\special{pn 8}%
\special{pa 600 400}%
\special{pa 600 550}%
\special{fp}%
\special{pn 8}%
\special{pa 600 650}%
\special{pa 600 800}%
\special{fp}%
\special{pn 8}%
\special{pa 1200 400}%
\special{pa 1200 800}%
\special{pa 1600 800}%
\special{pa 1600 400}%
\special{pa 1200 400}%
\special{fp}%
\special{pn 8}%
\special{pa 1200 800}%
\special{pa 1350 650}%
\special{fp}%
\special{pn 8}%
\special{pa 1450 550}%
\special{pa 1600 400}%
\special{fp}%
\special{pn 8}%
\special{pa 1400 400}%
\special{pa 1400 800}%
\special{fp}%
\special{pn 8}%
\special{pa 2400 400}%
\special{pa 2400 800}%
\special{pa 2800 800}%
\special{pa 2800 400}%
\special{pa 2400 400}%
\special{fp}%
\special{pn 8}%
\special{pa 2400 800}%
\special{pa 2600 400}%
\special{fp}%
\special{pn 8}%
\special{pa 2600 800}%
\special{pa 2800 400}%
\special{fp}%
\end{picture}}%

\begin{fct*}[{Turaev, \cite{Turaev}}]
$\Sk(\Sigma)$ is an $s$-deformation of 
the Goldman Lie algebra of $\Sigma$.
\end{fct*}

Recall that the Goldman Lie algebra \cite{Goldman}
is the Lie algebra whose underlying vector space 
is given by the space of isotropy classes of 
oriented loops on $\Sigma$
and whose Lie bracket is given by 
\[
 [\langle \alpha\rangle, \langle \beta\rangle]_{\text{Goldman}}
=\sum_{p \in \alpha \cap \beta} \pm \langle \alpha_p \beta \rangle.
\]
Here $\alpha:S^1 \to \Sigma$ is a loop on $\Sigma$,
and $\langle \alpha \rangle \in \pi^1(\Sigma)$ 
is its class in the fundamental group.
$\alpha\cap \beta$ is the set of intersection points,
and $\alpha_p\beta$ means a loop obtained by 
$\alpha$ and $\beta$ joined at $p$.
The sign $\pm$ is defined by the orientation of $\Sigma$
and the intersection behavior of $\alpha$ and $\beta$ at $p$.

The result due to Morton and Samuelson \cite{MS} is that 
the skein algebra for a torus is isomorphic to 
a specialization of elliptic Hall algebra.
So let us turn to explain the Hall algebras.

\subsection{Hall algebra}
\label{subsec:intro:Hall}

Let us briefly recall the theory of Hall algebra.
We refer \cite{Sc:lect} as a good review on the subject.

Let $\bbF_p$ is a finite field.
A finitary category $\catC$ 
is an $\bbF_p$-linear abelian category such that  
hom-spaces are of finite dimensional.
Denote by $\Iso(\catC)$ the set of isomorphism classes 
of objects in $\catC$.
The Ringel-Hall algebra \cite{Ringel} for $\catC$ is 
the $\bbQ$-vector space
\[
 \Hall(\catC) := 
 \{ f: \Iso(\catC)  \to \bbQ \mid \supp(f) < \infty \}
\]
with the multiplication 
\[
 f*g([M]) := \sum_{N \subset M}f([M/N])g([N]).
\]
Here we denoted by $[M]$ 
the isomorphism class of the object $M$.

If $\catC$ is hereditary, then by Green \cite{Green},
$\Hall(\catC)$ is a bialgebra with the comultiplication
\[
 \Delta(f)([M],[N]) := f([M \otimes N])
\]
together with the Hopf pairing 
\[
 \langle \delta_{[M]},\delta_{[N]} \rangle 
 := \delta_{M,N}/\#\Aut_{\catC}(M).
\]
Here we denoted by $\delta_{[M]}$ 
the characteristic function of $[M] \in \Iso(\catC)$.

Since we have a bialgebra with a Hopf pairing,
we can consider its  Drinfeld double
(see \cite{Sc:lect} and  \cite{J} for good accounts). 
For a hereditary and finitary abelian category $\catC$,
the Drinfeld double $\DHall(\catC)$
of $\Hall(\catC)$ is defined to be 
\[
 \DHall(\catC) := 
 \Hall(\catC) \otimes_{\bbQ} \Hall(\catC)
\]
as a vector space, 
and the multiplication is given by 
\[
 (m\otimes 1)\cdot(1 \otimes n) = m\otimes n,
 \quad
 \tsum \langle m_{(2)},n_{(1)}\rangle m_{(1)} \otimes n_{(2)}
=\tsum \langle m_{(1)},n_{(2)}\rangle 
 (1\otimes n_{(1)})\cdot(m_{(2)} \otimes 1).
\]
Here we used Sweedler's notation 
$\Delta(m)=\sum m_{(1)} \otimes m_{(2)}$.

Now we can state the result of Burban and Schiffmann \cite{BS}
on the Hall algebra for an elliptic curve.

\begin{fct}[{Burban-Schiffmann, \cite{BS}}]
\label{fct:BS}
Let $E$ be an elliptic curve over $\bbF_p$ and 
set $\catC := \Coh_{E}$, the category of coherent sheaves
over $E$.
Then
$\DHall(\catC)$ has a basis 
$\{u_{x} \mid x \in \bbZ^2\setminus\{(0,0)\}\}$
with relations
\\
\quad
(1)
$[u_x,u_y]=0$ if $x,y$ parallel
\\
\quad
(2)
$[u_x,u_y]=\pm\theta_{x+y}/\alpha_1$ 
if $\Delta(0,x,x+y)$ contains 
no integral points.
\\
Here $\theta_x$ is defined by
\[
 1+\sum_{k\ge1} \theta_{k z} w^k
 =\exp\Bigl(\sum_{n\ge1}\alpha_n u_{n z} w^n\Bigr)
\]
with
$\alpha_n:=(1-q^n)(1-t^{-n})(1-q^n t^{-n})/n$.
The symbols $q,t^{-1}$ denote the Weil numbers of $E$. 
\end{fct}

Recall that the Weil numbers $q,t^{-1}$ enjoys the properties
\begin{equation}\label{eq:qt-Weil}
 q/t=p,\quad
 \zeta_E(z)
:=\exp\Bigl(\sum \# E(\bbF_{p^n})\frac{w^n}{n}\Bigr)
 =\frac{(1-q z)(1-z/t)}{(1-z)(1-p z)}.
\end{equation}

Let us call the algebra $\DHall(\Coh_E)$ 
the \emph{elliptic Hall algebra}. 
Burban and Schiffmann also showed in \cite{BS} 
that $\DHall(\Coh_E)$ 
has an integral basis $\{w_x\}$, 
where each $w_x$ is proportional to $u_x$,
in the following sense.

\begin{dfn*}
Let $q,t$ be indeterminates.
Define $\EH_{q,t}$ to be
the algebra over $\bbZ[q^{\pm1/2},t^{\pm1/2}]$ generated by 
$w_x$'s with relations (1), (2) in Fact \ref{fct:BS}.
\end{dfn*}

We omit the proportional factor $w_x/u_x$.
The meaning of integral basis is that 
if we choose $q$ and $t$ to be the Weil numbers of $E$,
then we have 
\[
 \DHall(\Coh_E) \otimes_{\bbQ} \bbC \simeq 
 \EH_{q,t} \otimes_{\bbZ[q^{\pm1/2},t^{\pm1/2}]} \bbC.
\]

\subsection{The Morton-Samuelson isomorphism}

Now we can explain
the Morton-Samuelson isomorphism mentioned in the beginning.
Denote by $L_{r,d} \in \Sk(\torus)$
the class of loop winding $r$-times rightwards 
and winding $d$-times upwards,
as in $L_{0,1}$ and $L_{1,1}$ in \S\ref{subsec:skein}.

\begin{fct*}[{Morton-Samuelson, \cite{MS}}]
\begin{align*}
\Sk(\torus) &\longsimto
 \EH_{s^2,s^2} \otimes_{\bbZ[s^{\pm1}]} 
 \bbZ[s^{\pm1},v^{\pm1}],
\\
 L_{r,d} &\longmapsto \ 
 w_{(r,d)} \text{ if } \gcd(r,d)=1.
\end{align*}
\end{fct*}

Let us remark that 
$\EH_{q,t=q}$ is still well-defined since
$\theta_x/\alpha_1 = ([\gcd(x)]_{q^{1/2}})^2 u_x$.
For $t=q$, the proportional factor between $w_x$ and $u_x$
is given by 
\[
 w_x=(q^{d/2}-q^{-d/2})u_x
\] 
with $d:=\gcd(x)$.

\subsection{The result of this note}
\label{subsec:intro:outline}

Morton and Samuelson showed that the isomorphism 
between the skein algebra and the elliptic Hall algebra 
``by hand".
However, as they mentioned 
in the introduction of their paper \cite{MS},
it is natural to invoke 
homological mirror symmetry for torus/elliptic curve
\cite{K,PZ}:
\[
 \DFuk(\torus) \simeq \DCoh(E), \quad
 L_{1,d} \longleftrightarrow  \calL_d.
\]
Here $L_{1,d}$ is the Lagrangian submanifold 
corresponding to the loop $L_{1,d}$,
and $\calL_{d}$ is the line bundle over 
the elliptic curve $E$ defined over $\bbC$.
One may guess that the Morton-Samuelson isomorphism  
comes from equivalence of categories.

However we have one drawback:
the algebra $\EH_{s^2,s^2}$ would correspond to 
the elliptic curve $E$ with Weil numbers $q=t=s^2$.
Then  by \eqref{eq:qt-Weil}
$E$ is defined over $\bbF_p$ with $p=q/t=1$.
So in the present situation $E$ 
should be ``defined over $\bbF_1$".
Thus some theory of schemes over $\bbF_1$ will be needed.

The purpose of this paper is to give 
a construction of $\EH_{q,q}$ as the Hall algebra of 
``elliptic curve of $\bbF_1$".
Precisely speaking, we argue
\begin{description}
\item[B1]
Build a category $\catB$ using 
the \emph{monoidal Tate curve} $\wh{E}$.
$\wh{E}$ is a monoidal scheme over the formal monoidal scheme 
$\Spf\mnd{q}$,
which is seen as the space of the parameter $q$.
The category $\catB$ will be 
considered as the category of 
``coherent shaves over $E/\bbF_1$".

\item[B2]
The category $\catB$ is not abelian, and it is not 
even additive.
However it is a belian and quasi-exact category 
in the sense of Deitmar \cite{D:bel,D:qext}.
Then Szczesny's construction  of Hall algebra \cite{Sz}
can be applied, and we have an algebra $\DHall(\catB)$.

\item[B3]
Check $\DHall(\catB) \simeq \EH_{q,q}$.
\end{description}
The step \textbf{B3}, or Theorem \ref{thm:main},
is the main result of this note.

The steps \textbf{B1}--\textbf{B3} form the ``B-side"
of our strategy of the proof of Morton-Samuelson isomorphism
in terms of homological mirror symmetry.
For the completeness of explanation,
let us explain the remained ``A-side".
\begin{description}
\item[A1]
Build another category $\catA$ associated to a torus, 
which may be seen as the ``Fukaya category of tori over $\bbF_1$".
$\catA$ should have a parameter $s$ parametrizing 
the symplectic form, or the area form, on the torus.
\item[A2]
Consider Hall algebra $\Hall(\catA)$
\item[A3]
Check $\DHall(\catA)\simeq \Sk(\torus)$.
\end{description}

The Morton-Samuelson isomorphism will be the consequence of
\begin{description}
\item[HMS]
Show $\catA \simeq \catB$.
\end{description}
One may prove the step \textbf{HMS} by 
taking an appropriate $\bbF_1$-analogue of 
the mirror symmetry for torus/elliptic curve 
over $\bbZ$, shown by \cite{Gross,LP}.

\section{Hall algebra in the monoidal setting}
\label{sect:Hall}

In this section we explain Szczesny's definition of 
Hall algebra for monoid representation
and restate it in the setting of 
quasi-exact and belian category 
in the sense of Deitmar.

\subsection{Category of modules over commutative monoid}
\label{subsec:A-mod}

Let $A$ be a (multiplicative) commutative monoid with 
the absorbing element $0$, i.e.,
$a 0 = 0 a = 0$ for any $a \in A$.

An $A$-module is a pointed set $(M,*)$ with 
$A$-action $\cdot: A \times M \to M$ 
such that $0 \cdot m = *$ for any $m \in M$.
We will denote $M$ for $(M,*)$ for simplicity.

Morphisms, submodules, images 
are defined as in the case of modules over a group.

Let $M$ be an $A$-module and $N \subset M$ be a sub $A$-module.
The quotient $M/N$ is defined to be the module 
with the underlying pointed set 
$(M\setminus N)\sqcup \{*\}$.

We denote by $A$-$\modc$ the category of $A$-modules.
It is not abelian since 
$\coim(f) \to \im(f)$ 
is not an isomorphism in general.
It is not even additive.

$A$-$\modc$ has the zero object $\{*\}$
in the sense that for every $A$-module $M$
there exists a unique $A$-morphism $M \to \{*\}$ 
and a unique $\{*\} \to M$.
We will denote $\{*\}$ by $0$ hereafter.
Now we have the notion of kernels and cokernels of morphisms.

For $A$-modules $M$ and $N$,
is defined to be 
$M \sqcup N/\sim$ with $\sim$ 
identifying base-points.
Here $\sqcup$ denotes the (set-theoretic) disjoint sum.
Then $M \oplus N$  is the coproduct in the categorical sense.

The product 
$M \times N$ in the categorical sense 
is defined to be the set-theoretic product with diagonal action.

One can also check that 
$A$-$\modc$ has finite limits and colimits.


Let us state a property of $A$-$\modc$ 
which is necessary to prove the associativity
of Hall algebra.

\begin{lem}\label{lem:noether}
For an $A$-module $M$ and its sub $A$-module $N \subset M$,
there is an inclusion-preserving correspondence between 
$A$-modules $N \subset L \subset M$ 
and sub $A$-modules of $M/N$ given by $L \mapsto L/N$.
\end{lem}

\subsection{Hall algebra of monoid representations}

Let $A$ be a commutative monoid 
as in the previous subsection.
Since $A$-$\modc$ is not abelian,
one cannot expect to construct Hall algebra.
The idea of Szczesny \cite{Sz} is to restrict 
the class of morphisms 
so that we have a good family of exact sequences.

\begin{dfn}\label{dfn:normal}
A morphism $f$ in $A$-$\modc$ is called 
\emph{normal} 
if $\coim(f) \simto \im(f)$.
Denote by $A$-$\modc^n$
the subcategory of $A$-$\modc$ 
with the same objects 
and only normal morphisms. 
\end{dfn}

Now in the category $A$-$\modc$ 
we have the notion of exact sequence consisting 
of normal morphisms,
and can construct Hall algebra as in the abelian category 
(recall \S\ref{subsec:intro:Hall}).

\begin{fct}[{Szczesny, \cite[\S3]{Sz}}]
\label{fct:Hall:Sz}
Denote by $\Iso(A\text{-}\modc^n)$
the set of isomorphism classes of objects in $A$-$\modc^n$.
Then the $\bbQ$-vector space
\begin{equation}\label{eq:Hall:Amod}
 \{f:\Iso(A\text{-}\modc^n) \to \bbQ \mid 
   \#\supp(f)<\infty\}
\end{equation}
is a bialgebra with the multiplication $*$ 
and the comultiplication $\Delta$ defined as 
\[
 f*g([M]):=\sum_{N \subset M}f([M/N])g([N]), \quad 
 \Delta(f)([M],[N]) := f([M \otimes N]).
\]
Here we denoted by $[M]$ the class of the object $M$.
\end{fct}

As mentioned before, the associativity is proved using
Lemma \ref{lem:noether}.

\begin{dfn}\label{dfn:Hall:Sz}
We denote this bialgebra by $\Hall(A\text{-}\modc^n)$.
\end{dfn}

\subsection{Hall algebra for quasi-exact category}

According to \cite{Sz},
the category $A$-$\modc$ is an example of 
quasi-exact and belian category in the sense of Deitmar
\cite{D:qext, D:bel}.
or an example of proto-exact category in the sense of 
Dyckerhoff-Kapranov.
Let us restate the result in the previous subsection 
in such a categorical setting as we need in the latter argument.

\begin{dfn}
Let $\catC$ be a small category.
\\
(1)
$\catC$ is called \emph{balanced}
if every morphism
which is both monomorphism and epimorphism 
is an isomorphism.
\\
(2)
$\catC$ is called \emph{pointed} if it has an object $0$
such that for every object $X$ the sets $\Hom_{\catC}(X,0)$ 
and $\Hom_{\catC}(0,X)$
of morphisms have exactly one element respectively.
Such $0$ is unique up to unique isomorphism,
and called the \emph{zero object}.
The unique element in $\Hom_{\catC}(X,0)$ 
and $\Hom_{\catC}(0,X)$ are called the \emph{zero morphism}.
\end{dfn}

For a pointed category, 
one can introduce  the notion of kernels and cokernels
using the zero morphism.

\begin{dfn}[{Deitmar \cite{D:bel}}]
A \emph{belian} category is a balanced pointed category 
which has  finite products, kernels and cokernels,
and has the property that every morphism with zero cokernel 
is an epimorphism.
\end{dfn}

For a commutative monoid $A$,
the category $A$-$\modc$ is an example belian category.
In fact, one can check easily that 
$A$-$\modc$ is balanced, and the other axioms are already 
discussed in \S\ref{subsec:A-mod}.  

Next we turn to quasi-exact categories.
Let $\catC$ be a balanced pointed category.
As usual, we have the notion of short exact sequence
\[
 0 \to X \to Y \to Z \to 0
\] 
of morphisms in $\catC$.
We will call $X \to Y$ the kernel 
and $Y \to Z$ the cokernel of the short exact sequence.

\begin{rmk*}
In \cite{D:qext,D:bel} 
Deitmar called our exact sequences 
strong exact sequences.
\end{rmk*}


\begin{dfn}[{Deitmar \cite{D:qext}}]
A quasi-exact category is a balanced pointed category $\catC$
together with a class $\calE$ of short exact sequences 
such that 
\begin{itemize}
\item
for any two objects $X,Y$, the natural sequence 
$0 \to X \to X \oplus Y \to Y \to 0$ 
belongs to $\calE$,
\item
kernels in $\calE$ is closed under composition 
and base-change by cokernels in $\calE$,
\item
cokernels in $\calE$ is closed under composition 
and base-change by kernels in $\calE$.
\end{itemize}
\end{dfn}

In the construction of Szczesny's Hall algebra,
we considered the category $A$-$\modc^n$ of $A$-modules 
with normal morphisms.
One can check that $A$-$\modc^n$ together with 
short exact sequences consisting of normal morphisms
is a quasi-exact category.

Now we can restate Szczesny's result (Fact \ref{fct:Hall:Sz}).

\begin{dfn}
A category $\catC$ is called finitary if
for any two objects $X$ and $Y$,
the set $\Hom_{\catC}(X,Y)$ is finite
\end{dfn}

Let us denote a kernel $N \inj M$ in a quasi-exact category 
by $N \subset M$.

\begin{prop}
Let $\catC$ be a finitary, belian and quasi-exact category
with the class $\calE$ of short exact sequences.
Then the $\bbQ$-vector space
\begin{equation}\label{eq:Hall:space}
\{f:\Iso(\catC) \to \bbQ \mid  \#\supp(f)<\infty\}
\end{equation}
is a unital associative algebra with the multiplication
\begin{equation}\label{eq:Hall:mul}
f*g([M]):=\sum_{N \subset M}f([M/N])g([N]),
\end{equation}
with $[M]$ the class of the object $M$.
\end{prop}

The only non-trivial point is 
the associativity of the multiplication,
which is proved by using Lemma \ref{lem:noether}.
Let us restate it in the present context.
For a quasi-exact category $\catC$ with the class $\calE$
of short exact sequences, we denote a kernel $X \to Y$ 
appearing in a sequence of $\calE$ by $X \inj Y$,
and denote a cokernel $Y \to Z$ 
appearing in a sequence of $\calE$ by $Y \surj Z$.

\begin{lem}
Let $\catC$  a belian and quasi-exact category.
Suppose that we are given a commutative exact diagram 
with solid arrows in the next line.
\[
 \xymatrix{
  L \ar@{^{(}->}[r] \ar@{=}[d] & 
  M \ar@{->>}[r]    \ar@{^{(}->}[d] & 
  X \ar@{^{(}.>}[d] \\
  L \ar@{^{(}->}[r] &
  N \ar@{->>}[r]    \ar@{->>}[d] &  
  Y \ar@{.>>}[d] \\
  & Z \ar@{=}[r] & Z
 }
\]
Then there are dotted morphisms making the whole diagram 
commutative and exact.
\end{lem}

The proof of the lemma can be done by a standard diagram chasing.

Next we turn to coalgebra and bialgebra structures.
The coalgebra structure is defined in a usual way.

\begin{lem}
Let $\catC$ be a finitary belian quasi-exact category.
The vector space \eqref{eq:Hall:space} has a structure of
counital coassociative coalgebra 
with the comultiplication
\begin{equation}\label{eq:Hall:comul}
\Delta(f)([M],[N]):=f([M \oplus N]).
\end{equation}
\end{lem}

Now to construct a bialgebra, we need the hereditary condition
on our category $\catC$.

\begin{dfn}
A belian quasi-exact category $\catC$ is called hereditary
if $\Ext_{\catC}^n(X,Y)=0$ for any objects $X,Y$ 
and any $n \in \bbZ_{\ge2}$.
\end{dfn}

Here the higher extension $\Ext^n$ is defined 
by the derived functor of $\Hom$.
The theory of derived functors for a belian category 
is given in \cite[\S1.6]{D:bel}, and we will not repeat it.

Now the argument \cite{Green} 
(see also \cite[\S1.5]{Sc:lect})
on the Hall bialgebra for abelian category 
can be applied to our setting, and we have

\begin{thm}
Let $\catC$ be a belian quasi-exact category
which is finitary and hereditary.
The vector space \eqref{eq:Hall:space} 
has a structure of bialgebra
with the multiplication \eqref{eq:Hall:mul}
and the comultiplication \eqref{eq:Hall:comul}.
\end{thm}

\section{Monoidal schemes, monoidal Tate curve and Hall algebra}

In this section we introduce the monoidal Tate curve $\wh{E}$.
The desired category $\catB$ 
mentioned in \S\ref{subsec:intro:outline}
will be the constructed as a category of sheaves over $\wh{E}$.
Our monoidal Tate curve is a counter part of
the standard Tate curve in the $\bbF_1$-scheme setting.
We will use Deitmar's theory \cite{D} of monoid schemes 
as $\bbF_1$-scheme theory.

\subsection{Monoidal schemes}

We follow Deitmar \cite{D} and use \emph{monoidal schemes}
as the theory of schemes  over $\bbF_1$.

For a commutative monoid $A$,
an ideal $\frka$ is a subset such that $\frka A \subset \frka$.
One can define prime ideals and localizations 
as in the commutative ring case.

Now an affine monoidal scheme 
$\Spec^{\mon}(A)$ is the set of prime 
ideals of $A$ with the Zariski topology.
A monoidal schemes is a topological space $X$ 
with a sheaf 
$\shO^{\mon}_X$ of monoids,
locally isomorphic to some $\Spec^{\mon}(A)$.

Given a monoidal scheme $X$,
we can define 
$\shO^{\mon}_X$-modules, 
(quasi-)coherent $\shO^{\mon}_X$-modules
and locally free $\shO^{\mon}_X$-modules
in a quite similar was as the usual scheme case. 
We also have the notion of 
a perfect complex of $\shO^{\mon}_X$-modules,
i.e., a cohomologically bounded complex 
of coherent $\shO^{\mon}_X$-modules 
which is locally quasi-isomorphic to a bounded complex 
of locally free $\shO^{\mon}_X$-modules of finite ranks.

Relative notions can also be introduced to monoidal schemes.
In particular,
we have the notion of flat family of monoidal schemes,
i.e., a flat morphism $X \to S$ of monoidal schemes. 
Let us explain the definition of flatness:
A commutative monoid $B$ over another $A$ 
is flat if the tensor product functor 
$B \otimes_A (-): A\text{-}\modc^n \to B\text{-}\modc^n$
is exact.

We also have a functor 
\begin{equation}\label{eq:XtoX_Z}
 X \longmapsto X_{\bbZ}
\end{equation}
mapping a monoidal scheme $X$ to 
a scheme $X_{\bbZ}$ over $\bbZ$.
It is induced by the functor $A \mapsto \bbZ[A]$ from 
a commutative monoid $A$ to the monoidal ring $\bbZ[A]$ of $A$.

\subsection{Monoidal Tate curve}

Now let us recall the Tate curve $\wh{E}_{\text{Tate}}$
in the sense of usual scheme, 
following \cite[\S8.4.1]{Gross} and \cite[\S9.1]{LP}.

Consider the following toric data.
\begin{equation}\label{eq:tate:toric}
\begin{split}
&\rho_i := \bbQ_{\ge0}(i,1) \subset \bbQ^2 \ (i \in \bbZ),
 \quad
 \rho_i^\vee 
 := \{x \in \bbQ^2 \mid \langle x,\rho_i\rangle 
      \subset \bbQ_{\ge 0}\}
  = \{(x_1,x_2) \mid i x_1+x_2 \ge 0 \},
 \\
&\sigma_{i+1/2}^\vee := \rho_i^\vee \cap \rho_{i+1}^\vee.
\end{split}
\end{equation}

For a commutative ring $R$, 
set
$U_{i+1/2}:=\Spec R[\sigma_{i+1/2}^\vee \cap \bbZ^2]$.
These affine schemes glue to give a scheme $E$ over $R$
(see \cite[\S9.1.1]{LP} for the explicit form of this gluing map).
The map
$\sigma_{i+1/2}^{\vee} \to \bbQ$, $(x,y) \mapsto y$
gives a morphism $E_{\text{Tate}} \to \Spec R[q]$.
The Tate curve 
$\wh{E}_{\text{Tate}} \to \Spf R[[q]]$
is a formal thickening of this morphism.

The Tate curve $\wh{E}_{\Tate}$ can be considered 
as a family of elliptic curves over $R$ in the following sense.
$\bbZ$ acts on $E_{\Tate}$ in the way 
that $1 \in \bbZ$ corresponds to 
the matrix $\begin{pmatrix}1&1 \\ 0&1\end{pmatrix}$ 
acting on $(\bbQ^2)^{\vee}$.
Taking $R=\bbC$, we can describe this action on the big 
torus $(\bbC^*)^2 \subset E_{\text{Tate}}$ as 
$(z,q)\mapsto (z q, q)$ with $z,q$ corresponding to 
the basis $(1,0),(0,1)$ of $\bbQ^2$.
If we fix $q \in \bbC \setminus\{0\}$ and consider the fiber 
of the map $E_{\text{Tate}} \to \Spec \bbC[q]$,
then the $\bbZ$-action on this fiber is 
generated by $z \mapsto z q$, 
and the quotient is identified with $\bbC^*/q^{\bbZ}$.

Now we construct the \emph{monoidal Tate curve} $\wh{E}$ 
defined by replacing rings in the 
construction of $\wh{E}_{\text{Tate}}$ by monoids.
So we consider the toric data \eqref{eq:tate:toric}
and set
$U_{i+1/2}^{\mon}:=\Spec^{\mon} 
 \langle \sigma_{i+1/2}^\vee \cap \bbZ^2\rangle$.
Gluing and formal thickening give
$\wh{E} \to \Spf^{\mon}\mnd{q}$,
where
$\Spf^{\mon}\mnd{q}$ is the
(formal) monoidal scheme of $\mnd{q} := q^{\bbN}$.
As mentioned above in the case of $\wh{E}_{\text{Tate}}$,
the monoidal Tate curve $\wh{E}$ has a $\bbZ$-action.

Recall the functor \eqref{eq:XtoX_Z} mapping 
each monoidal scheme to a scheme over $\bbZ$.
By construction we have 
\begin{equation}\label{eq:EtoEtate}
\wh{E}_{\bbZ} \simeq \wh{E}_{\Tate}
\end{equation}
 
For later use we consider the zeta function of $\wh{E}$
using Deitmar's result \cite{D:qext}.
The (local) zeta function of a monoidal scheme $X$
is defined to be
$\zeta_{X}(z):=\exp(\sum_{n\ge1}\#X_{\bbZ}(\bbF_{p^n})z^n/n)$
with $p$ a prime number.
Then by \eqref{eq:EtoEtate} and by the fact 
that the Tate curve is an elliptic curve over $\bbZ$,
we have 
\begin{equation}\label{eq:zeta}
 \zeta_{\wh{E}/\mnd{q}}(z)=\frac{(1-q z)(1-z/q)}{(1-z)^2}.
\end{equation}

\section{Hall algebra of monoidal Tate curve}

We now introduce the category $\catB$ 
outlined in \S\ref{subsec:intro:outline},
and define the Hall algebra associated to it
using the general theory in \S\ref{sect:Hall}.

\subsection{Category of sheaves over monoidal scheme}

We want to consider the category of 
$\shO^{\mon}_{\wh{E}}$-modules, i.e. 
sheaves over the monoidal Tate curve.
An $\shO^{\mon}_{\wh{E}}$-module 
consists of the modules $M_{i+1/2}$
over the commutative monoids 
$\langle \sigma^{\vee}_{i+1/2} \cap \bbZ^2\rangle$
and of the gluing data.

As in the case of the category $A$-$\modc$ 
of monoid representations,
the category of $\shO^{\mon}_{\wh{E}}$-modules 
with normal morphisms
is a belian quasi-exact category.
We can restrict the consideration to 
coherent $\shO^{\mon}_{\wh{E}}$-modules,
and denote by $\Coh_{\wh{E}}$ 
the belian quasi-exact category of 
$\shO^{\mon}_{\wh{E}}$-modules with normal morphisms,
which is obviously finitary and hereditary.

\subsection{Fourier-Mukai transform}

Before introducing some $\shO^{\mon}_{\wh{E}}$-modules,
let us recall the standard facts on the sheaves over
elliptic curves in the usual scheme theory.
Hereafter the word ``a sheaf over a scheme $X$" means 
a $\shO_X$-module,
and the word ``a vector bundle" means a locally free sheaf.

Recall that for a smooth projective curve $C$
over a field  any coherent $\shO_C$-module 
has the Harder-Narasimhan filtration,
and the associated factors are semi-stable sheaves.
Each semi-stable sheaf has 
a Jordan-H\"{o}lder filtration 
and the associated factors are stable sheaves. 
Let us denote by $\Coh^{ss}_C(r,d)$ 
the category of coherent semi-stable $\shO_C$-modules
with rank $r$ and degree $d$. 
Denote also by $\Sky_C(d)$ 
the category of skyscraper sheaves on $C$ of length $d$.
We have $\Sky_C(d)=\Coh^{ss}_C(0,d)$.

In the case of an elliptic curve $C=E$,
the Fourier-Mukai transform gives the equivalence 
of categories
\begin{equation}\label{eq:FMT:S}
 \Coh^{ss}_E(r,d) \simto \Coh^{ss}_E(d,-r)
 \text{ if } d>0,
 \quad 
 \Coh^{ss}_E(r,0) \simto \Sky_E(r) 
 \text{ if } r>0.
\end{equation}
The tensor product functor $- \otimes_{\shO_E} L$
with a degree one line bundle $L$ 
gives the equivalence
\begin{equation}\label{eq:FMT:T}
 \Coh^{ss}_E(r,d) \simto \Coh^{ss}_E(r,d+r).
\end{equation}

These equivalences gives a complete picture of 
the category $\Coh_E$ of coherent sheaves over $E$.
We refer \cite[Chap.\ 3]{BBH} as a nice account.
Let us recall a part of the description of $\Coh_E$.
All stable vector bundles are 
simple semi-homogeneous vector bundles.
Here a semi-homogeneous vector bundle is a sheaf
of the form $\pi_* L$ with some isogeny $\pi: E' \to E$
and a line bundle $L$ on $E'$.

\subsection{Fourier-Mukai transform in the monoidal setting}

Now let us turn to the monoidal Tate curve $\wh{E}$.
The notions of relative locally free sheaves
is well-defined in the monoidal setting.
We also have the notion of relative divisors,
so that the rank and degree of relative sheaves are well-defined.
As in the case of the usual elliptic scheme,
we have the relative Jacobian monoidal scheme 
$\Jac_{\wh{E}/\mnd{q}}$
parametrizing the degree 0 line bundles over $\wh{E}/\mnd{q}$
with the universal family $\calP$ over 
$\wh{E} \times \Jac_{\wh{E}/\mnd{q}}$.
The sheaf $\calP$ is nothing but the Poincar\'{e} bundle.
As in the usual case,
we have a natural isomorphism
$\Jac_{\wh{E}/\mnd{q}} \simeq \wh{E}$.

Thus we have the Fourier-Mukai transform
$\frkS(-):=(\bfR\pi_1)_*(\pi_2^*(-)\otimes \calP)$
with $\pi_i$ the projection from 
$\wh{E} \times \Jac_{\wh{E}/\mnd{q}}$ to $i$-th factor.
Here we used the derived functors in the monoidal scheme setting
\cite{D:bel}.

We also denote by $\frkT(-):= L \otimes (-)$ 
the tensor product functor with 
the relative line bundle of degree $1$ on $\wh{E}/\mnd{q}$.

The functors $\frkS$ and $\frkT$ are derived equivalences,
and those actions 
are similar as in the usual scheme case \eqref{eq:FMT:S}
and \eqref{eq:FMT:T} if we appropriately change the definition 
of the,subcategories $\Coh_{\wh{E}}^{ss}$.
In particular,
the rank and degree of sheaves, denoted by $(r,d)$,
are changed to $(d,-r)$ and $(r,d+r)$ respectively.
In other words,
the group generated by $\frkS$ and $\frkT$ acts on 
the lattice $\bbZ^2$ of ranks and degrees by 
the natural $\SL(2,\bbZ)$ action.

\subsection{The category $\catB$ and Hall algebra}

Let us recall the Fourier-Mukai transform 
$\frkS$ introduced in the previous subsection.
We denote by $\Sky_{\wh{E}/\mnd{q}}(d)$ 
the category of relative skyscraper sheaves with length $d$
and normal morphisms.
We also denote by $\Sky_{\wh{E}/\mnd{q}}$ 
the full subcategory of $\Coh_{\wh{E}/\mnd{q}}$ 
consisting of $\Sky_{\wh{E}/\mnd{q}}(d)$'s 
for all $d \in \bbZ_{\ge0}$ and normal morphisms.

\begin{dfn}
Define $\catB$ to be the subcategory of
$\shO^{\mon}_{\wh{E}}$-modules
consisting of the image of $\Sky_{\wh{E}/\mnd{q}}$ 
under the repetitions of 
$\frkS^{\pm1}$ and $\frkT^{\pm1}$
with normal morphisms.
\end{dfn}

One can apply Szczesny's construction of Hall algebra 
(Definition \ref{dfn:Hall:Sz}) to $\catB$. 
Denote by $\Hall(\catB)$ the resulting algebra.
It is a bialgebra with Hopf pairing,
so we can consider its Drinfeld double,
denoted by $\DHall(\catB)$.

\begin{thm}\label{thm:main}
As associative algebras over $\bbZ[q^{\pm1}]$,
we have the isomorphism
\[
 \DHall(\catB) \simeq \EH_{q,q}.
\]
\end{thm}

One remark is in order.
The algebra $\DHall(\catB)$ is defined over $\bbQ$.
But in the following argument 
we show that it has an integral basis and 
can be seen defined over $\bbZ[q^{\pm1}]$.

Our argument is the same as \cite{BS},
except the point that 
we use $\frkS$ and $\frkT$ instead of 
the usual Fourier-Mukai transform
and the tensor product functor with line bundle.
We will not repeat the detailed discussion,
but explain some key steps.

First we consider the grading 
and commutative subalgebras of $\DHall(\catB)$
arising naturally from the geometric setting.
Consider the subalgebra 
generated by the class of objects in 
$\Sky_{\wh{E}/\mnd{q}}$.
We denote by $S(0,d)$ the class of 
the relative skyscraper sheaf 
with length $d$.
Recall also that $\delta_{L}$ 
is the characteristic function of $L$.

\begin{prop}\label{prop:csa}
The subalgebra $\Hall(\Sky_{\wh{E}/\mnd{q}})$ 
of $\DHall(\catB)$ 
is isomorphic to the polynomial algebra with 
infinite variables $\delta_{S(0,d)}$, $d\in\bbZ_{\ge1}$.
\end{prop}

\begin{proof}
The exact sequences in $\Sky_{\wh{E}/\mnd{q}}$ 
are all split,
so that the algebra 
$\Hall(\Sky_{\wh{E}/\mnd{q}})$ is commutative.
It is also clear that $S(0,d)$'s generate 
$\Hall(\Sky_{\wh{E}/\mnd{q}})$,
so the statement holds.
\end{proof}

The argument of \cite[\S4.1]{BS} or 
\cite[Chap.\ III]{Macd}
gives an explicit description of 
the commutative subalgebra $\Hall(\Sky_{\wh{E}/\mnd{q}})$. 
In fact, it is a bialgebra 
under the comultiplication \eqref{eq:Hall:comul},
and it is isomorphic to the classical Hall bialgebra.
In the present context, we have

\begin{prop}
The subalgebra $\Hall(\Sky_{\wh{E}/\mnd{q}})$
is isomorphic to the subalgebra of $\EH_{q,q}$
generated by $w_{(0,d)}$, $d \in \bbZ_{\ge1}$.
The isomorphism is given by 
$(q^{d/2}-q^{-d/2})\delta_{S(0,d)} \mapsto w_{0,d}$.
\end{prop}

By the description of $\EH_{q,q}$, this isomorphism 
enable us to consider the subalgebra 
$\Hall(\Sky_{\wh{E}/\mnd{q}})$ to be defined 
over $\bbZ[q^{\pm1}]$ with basis 
$(q^{d/2}-q^{-d/2})\delta_{S(0,d)}$.

Next, we note the following properties of our Hall algebra.

\begin{prop}
(1)
$\DHall(\catB)$ has a $\bbZ^2$-grading.
We will denote the associated decomposition as  
$\DHall(\catB)=\oplus_{(r,d) \in \bbZ^2}H^{r,d}$.
\\
(2)
The group $\SL(2,\bbZ)$ act on 
$\DHall(\catB)$ as algebra automorphisms,
which induces the natural $\SL(2,\bbZ)$-action on 
the grading $\bbZ^2$.
\\
(3)
For a coprime pair $(r_0,d_0) \in \bbZ^2$,
The submodule $\oplus_{n \in \bbZ} H^{n r_0, n d_0}$
is a commutative algebra isomorphic to 
the Drinfeld double of $\Hall(\Sky_{\wh{E}/\mnd{q}})$.
\end{prop}

\begin{proof}
(1) 
The rank and degree of a sheaf give the grading.
\\
(2)
This is a consequence of derived equivalences  
$\frkS$ and $\frkT$.
\\
(3)
The functors $\frkS$ and $\frkT$ 
give the standard 
$\SL(2,\bbZ)$-action on the grading $\bbZ^2$.
Note also that 
the Drinfeld double $\DHall(\Sky_{\wh{E}/\mnd{q}})$
is the same as $\oplus_{d \in \bbZ}H^{0,d}$.
Using the actions of $\frkS$ and $\frkT$,
one can map $\oplus_{d \in \bbZ}H^{0,d}$ to a given 
$\oplus_{n \in \bbZ} H^{n r_0, n d_0}$.
Since $\frkS$ and $\frkT$ preserve exact triangles,
we find that $\oplus_{n \in \bbZ} H^{n r_0, n d_0}$
is isomorphic to $\oplus_{d \in \bbZ}H^{0,d}$.
\end{proof}

The remaining part of the proof of Theorem \ref{thm:main}
is the construction of the algebraic homomorphism
$\varphi:\EH_{q,q} \to \DHall(\catB)$
sending generator $w_{(r,d)}$ 
with coprime $(r,d) \in \bbZ^2$ to $\delta_{S(r,d)}$,
where $S(r,d)$ is the class of a stable sheaf 
with rank $r$ and degree $d$.
In fact, the correspondence 
$w_{(r,d)} \mapsto \delta_{S(r,d)}$
determines the vector space homomorphism 
$\varphi$ uniquely,
since each object of $\catB$ has 
the Harder-Narasimhan filtration and 
the Jordan-H\"{o}lder filtration.
Then the argument in \cite[\S5]{BS} can be applied to 
our situation, showing that $\varphi$ is an algebra homomorphism.
Thus we have the isomorphism $\EH_{q,q} \to \DHall(\catB)$.


\subsection*{Acknowledgements.} 
The author is supported by the Grant-in-aid for 
Scientific Research (No.\ 16K17570), JSPS.
This work is also supported by the 
JSPS for Advancing Strategic International Networks to 
Accelerate the Circulation of Talented Researchers
``Mathematical Science of Symmetry, Topology and Moduli, 
  Evolution of International Research Network based on OCAMI''
and by the JSPS Bilateral Program
``Topological Field Theory and String Theory -- 
  from topological recursion to quantum toroidal algebras".

The author would like to thank Y.~Yonezawa for 
the stimulating discussion on skein algebras.


\end{document}